\documentclass[letterpaper, 10 pt, conference]{ieeeconf}  

\IEEEoverridecommandlockouts                              

\usepackage{xspace}
\usepackage{soul}
\usepackage{mathtools}
\usepackage{makecell}
\usepackage{cite}
\usepackage{nicefrac}
\usepackage{float}
\usepackage{wrapfig}
\usepackage{amsmath,amssymb,amsfonts}

\newcommand{\ra}{\rightarrow}

\newcommand{\Ra}{\Rightarrow}

\newtheorem{asm}{Assumption}
\newtheorem{prp}{Proposition}
\newtheorem{rem}{Remark}

\newcommand{\ie}{\unskip, i.\,e.,\xspace}
\newcommand{\eg}{\unskip, e.\,g.,\xspace}

\newcommand{\N}{\ensuremath{\mathbb{N}}}

\newcommand{\R}{\ensuremath{\mathbb{R}}}

\newcommand{\X}{\ensuremath{\mathbb{X}}}

\newcommand{\U}{\ensuremath{\mathbb{U}}}

\let\emptyset\varnothing

\newcommand{\eps}{\ensuremath{\varepsilon}}

\DeclareMathOperator*{\argmin}{arg\,min}

\DeclareMathOperator*{\arginf}{arg\,inf}

\newcommand{\spc}{\ensuremath{\,\,}}	

\makeatletter
\newcommand{\subalign}[1]{%
	\vcenter{%
		\Let@ \restore@math@cr \default@tag
		\baselineskip\fontdimen10 \scriptfont\tw@
		\advance\baselineskip\fontdimen12 \scriptfont\tw@
		\lineskip\thr@@\fontdimen8 \scriptfont\thr@@
		\lineskiplimit\lineskip
		\ialign{\hfil$\m@th\scriptstyle##$&$\m@th\scriptstyle{}##$\crcr
			#1\crcr
		}%
	}
}

\newcommand{\Qfnc}{\ensuremath{\mathcal{Q}}}

\title{\LARGE \bf
Model predictive control with stage cost shaping inspired by reinforcement learning
}

\author{Lukas Beckenbach, Pavel Osinenko and Stefan Streif
\thanks{The authors are with the Automatic Control and System Dynamics Laboratory,
        Technische Universit{\"a}t Chemnitz, 09107 Chemnitz, Germany,
        {\tt\small $\{$lukas.beckenbach,pavel.osinenko,} {\tt\small stefan.streif$\}$@etit.tu-chemnitz.de}%
}}

\begin{document}

\maketitle
\thispagestyle{empty}
\pagestyle{empty}

\begin{abstract}

This work presents a suboptimality study of a particular model predictive control with stage cost shaping based on the ideas of reinforcement learning.
The focus of the study is to derive quantities relating the infinite-horizon cost under the said variant of model predictive control to the respective infinite-horizon value function.
The basis control scheme involves usual stabilizing constraints comprising of a terminal set and a terminal cost in the form of a local Lyapunov function.
The stage cost is adapted using the principles of Q-learning, a particular approach to reinforcement learning.
The work is concluded by case studies with two systems for wide ranges of initial conditions.

\end{abstract}


\section{Introduction}
\label{sec:Introduction}

Briefly, MPC converts an infinite-horizon potential control problem to tractable, finite-horizon, online optimizations. 
Since several control setups are available for the corresponding finite-horizon optimization, different performances with respect to the original infinite-horizon problem can be expected. 
In particular, unconstrained MPC in the sense of minimizing solely a truncation of the infinite-horizon cost each time step without any stabilizing so-called terminal constraints was studied rigorously \cite{Jadbabaie2001a,Gruene2009,Reble2011}. 
Though the question of suboptimality under stabilizing terminal constraint was addressed in various setups \cite{Tuna2006,Gruene2008}, they require preliminary assumptions on the LF and/or the terminal set, of which some are reviewed and relaxed herein. 

\textbf{Problem formulation.} Consider in the following the discrete-time system
\vspace{-.5em}
\begin{align} \label{eq:sys}
x_{k+1}=f(x_k,u_k), \quad k=0,1,\dotsc, \; x_0 \in \X,
\vspace*{-.5em}
\end{align}
with state $x_k \in \X \subset \R^{n_x}$ and control $u_k \in \U \subset \R^{n_u}$. 
The map $f: \X \times \U \rightarrow \X$ is continuous in both arguments and $(\bar{x},\bar{u})=(0,0)$ is the unique equilibrium point \ie $f(x,u) = x$ if and only if $x = \bar{x}$ and $u=\bar{u}$.

In this work, deterministic and fully available dynamics are considered, whereas there exist approaches of adapting the model \eg by neural networks \cite{Liu2012}. 

With \eqref{eq:sys}, the infinite-horizon optimal control problem (IHOCP) amounts to solving
\vspace{-.5em}
\begin{align} \label{eq:cost-opt-hor-inf}
\begin{split}
\forall x_0 \in \X: \; V^{\infty}(x_0):= \; \min_{\{u_k\}_{k=0}^{\infty}} \; &\sum_{k=0}^{\infty} l(x_k,u_k), \; \text{s.t.} \; \eqref{eq:sys} 
\end{split}
\end{align}
where the stage cost (or state-input penalty) has the form $l: \X \times \U \rightarrow \R_{\geq 0} $, with $l(\bar{x},\bar{u}) = 0$, $l(x,u)>0$ for $(x,u) \neq (\bar{x},\bar{u})$ and $l(x,\bar{u}) \leq l(x,u)$ for any $x$ and $u$. 
Denote $\U_{\text{ad}}$ the admissible set, containing control sequences $\{u_k\}_{k=0}^{\infty}$, with $u_k \in \U$ and $u_k \ra \bar{u}$ for $k \ra \infty$, for which $f(x_k,u_k) \in \X$, $f(x_k,u_k) \ra \bar{x}$ for $k \ra \infty$ and $\sum_{k=0}^{\infty} l(x_k,u_k) < \infty$. 
It is assumed that $\U_{\text{ad}} \neq \emptyset$. 
The minimizer $\{u_k^{\infty}\}_{k=0}^{\infty} \in \U_{\text{ad}} $ associated with \eqref{eq:cost-opt-hor-inf} may be given in a closed form as a feedback \ie there exists a function $\mu^{\infty}: \X \rightarrow \U$ such that $u_k^{\infty}=\mu^{\infty}(x_k)$. 

Let $F:\X_f \rightarrow \U$ and $\X_f \subset \X$ containing $\{\bar{x}\}$ in the interior. 
A nonlinear MPC (NMPC) optimization problem for time $k$, with terminal cost and terminal set constraint, may be given by
\vspace{-.3em}
\begin{subequations} \label{eq:MPC-1}
	\begin{align}
	\tilde{V}_{\text{MPC}}^{\infty}(x_k)& :=   \nonumber \\
	\min_{\{u_k(i)\}_{i=0}^{N-1}} \; &\sum_{i=0}^{N-1} l(x_k(i),u_k(i)) + F(x_k(N)) \\
	\text{s.t.} \quad &x_k(i+1)=f(x_k(i),u_k(i)), \; x_k(0)=x_k \\
	\quad &x_k(i) \in \X, \; u_k(i) \in \U \label{eq:MPC-1-xu-constr}\\
	\quad &x_k(N) \in \X_f  \label{eq:MPC-1-xstab-constr}.
	\end{align}
\end{subequations}
Constraints in state and input in \eqref{eq:MPC-1-xu-constr} are commonly included whereas \eqref{eq:MPC-1-xstab-constr} is the so-called stabilizing constraint. 
Certain assumptions on the terminal cost $F$ -- namely, $F$ being a local LF with controller $\mu_F$ such that $x \in \X_f \Ra f(x,\mu_F(x)) \in \X_f$ and $\Delta_{\mu_F(x)} F(x)$ having a specific decay rate -- allows relating $\tilde{V}_{\text{MPC}}^{\infty}(x_k)$ to the IHOCP. 

The main drawback of this approach is that the state sequence is forced to converge to $\X_f$ after $N$ steps, which may render the minimizer of \eqref{eq:MPC-1} possibly suboptimal with respect to the first $N$ elements of the infinite sum in \eqref{eq:cost-opt-hor-inf}.

\textbf{Q-learning-inspired stage cost shaping.} The aim of this work is to investigate a particular adaptation scheme of the stage cost in the MPC scheme \eqref{eq:MPC-1} in the form $l \mapsto c \cdot l$. 
This specific approach is motivated by the fact that online adaptation or learning, in general, may use only limited or single (the current) state-input values $(x_k,u_k)$. 
Although ``buffers'', as it was suggested in \cite{Lillicrap2016}, can be utilized, the presented adaptation scheme uses solely the current data. 
Other methods of incorporating learning into MPC include data-based constraint design for robust MPC \cite{Zanon2019} or terminal cost adaptation in MPC through means of iterative learning control \cite{Rosolia2018}. 
The adaptation scheme is based on the principles of Q-learning as follows. 
Interpreting $c \cdot l$ as an approximation to the Q-function
\begin{align} \label{eq:Q-fcn}
\forall (x,u) \in \X \times \U : \; \Qfnc(x,u) = l(x,u) + \min_{a \in \U} \; \Qfnc(f(x,u),a),
\end{align}
the coefficients $c$ should be updated so as to minimize the \emph{temporal difference} \cite{Sutton1988}. 
More details on this approach are given in the subsequent sections. 
The aim of the work is to derive suboptimality bounds of the respective MPC scheme under such an adaptation of the stage cost. 
Particular realizations of the coefficient update rule are discussed in Section \ref{subsec:coeff-interp}. 
The case studies of Section \ref{sec:simstudy} present comparison of the baseline MPC with the suggested stage-cost-shaping MPC for wide ranges of initial conditions.

\textbf{Notation.} The natural numbers including zero are denoted by $\N_{0} := \N \cup \{0\}$. For any scalar functions $W_1(x_k,k):\R^{n_x} \times \N_0 \rightarrow \R_{\geq 0}$ and $W_2(x_k):\R^{n_x} \rightarrow \R_{\geq 0}$, denote $\Delta_u W_1(x_k,k):= W_1(f(x_k,u),k+1)-W_1(x_k,k)$ and $\Delta_u W_2(x_k):= W_2(f(x_k,u))-W_2(x_k)$ the difference of subsequent function values along the state recursion \eqref{eq:sys} under some $u \in \R^{n_u}$.


\section{Suboptimality Description and Algorithm}
\label{sec:Objective_Algorithm}

In the following, some key assumptions on the function $F$ are introduced, one of which can be found frequently in MPC literature \cite{Fontes2001,Jadbabaie2005}. 
Furthermore, optimality and stability related statements are presented, summarizing the objective of this work. 
Specifically, a suboptimality comparison is suggested for the case of the predictive scheme under parametrized decay of a local LF \eqref{eq:MPC-1}.

\subsection{Central Assumptions and Suboptimality Estimate}

Essentially, a key ingredient to ensure stability is the existence of a local controller inside the set, which the terminal predicted state is sought to reach. This means:

\begin{asm} \label{asm:local-ctrl}
	There exists a local controller $\mu_F:\X_f \rightarrow \U$ and a local LF $F:\X_f \rightarrow \R_{\geq 0}$ such that
	\vspace*{-0.5em}
	\begin{align*}
	\forall x_k \in \X_f: \; \Delta_{\mu_F(x_k)} F(x_k) \leq -\alpha_F(x_k).
	\end{align*}
\end{asm}

It is not until later in the analysis that the function $F(x)$ is required to have a decay of magnitude of the stage cost \cite{Chen1998,Hu2002,Jadbabaie2005}:
\begin{asm} \label{asm:local-ctrl-2}
	There exists a local controller $\mu_F: \X_f \rightarrow \U$ and a local LF $F: \X_f \rightarrow \R_{\geq 0}$ such that
	\vspace*{-0.5em}
	\begin{align*}
	\forall x_k \in \X_f: \; \Delta_{\mu_F(x_k)} F(x_k) =  - \bar{\alpha}_k l(x_k,\mu_F(x_k)), 
	\end{align*}
	where $1 \leq \bar{\alpha}_k < \infty$, for all $k \in \N_0$.
\end{asm}
The function $F$ in Assumption \ref{asm:local-ctrl-2} can be regarded as the infinite tail inside $\X_f$ under the stage cost $l$, obtained \eg by local LQ analysis \cite{Chen1998}. 

\begin{prp} \label{prop:cost-rel}
	Let $W: \X \rightarrow \R_{\geq 0}$ be a LF on $\X$ \ie $W(\bar{x})=0$ and positive for any $x \in \X - \{\bar{x}\} $, and $\mu: \X \rightarrow \U$ the control policy, with $\mu(\bar{x}) = \bar{u}$ and $\{\mu(\tilde{x}_k)\}_{k=0}^{\infty} \in \U_{\text{ad}}$, where $\tilde{x}_{k+1}=f(\tilde{x}_k,\mu(\tilde{x}_k))$ under some $\tilde{x}_0 \in \X$.
	Furthermore, let $W$ and $\mu$ be such that
	\begin{align} \label{eq:decay-modified}
	\forall k \in \N_0: \; \Delta_{\mu(\tilde{x}_{k})} W(\tilde{x}_{k}) \leq -c_k l(\tilde{x}_{k},\mu(\tilde{x}_{k})), 
	\end{align}
	for some sequence $\{c_k\}_{k=0}^{\infty}$, with $c_k \in (0,\bar{c})$, $0 < \bar{c}<\infty$. 
	Denote $\delta_0 = W(x_0) - V^{\infty}(x_0)$ for any $x_0 \in \X$ and $(\Delta l)_k := c_k l(\tilde{x}_{k},\mu(\tilde{x}_{k})) - l(\tilde{x}_{k},\mu(\tilde{x}_{k}))$. 
	It holds that
	\begin{align} \label{eq:cost-rel-mod-decay-VN}
	\sum_{k=0}^{\infty} l(\tilde{x}_{k} , \mu(\tilde{x}_{k} )) \leq V^{\infty}(x_0) + \delta_0 - \sum_{k=0}^{\infty} (\Delta l)_k,
	\end{align}
	where furthermore
	\begin{align} \label{eq:sum-Delta-c-bound}
	\sum_{k=0}^{\infty} (\Delta l)_k \leq \delta_0.
	\end{align}
\end{prp}

\begin{proof}
	By \eqref{eq:decay-modified}, for any $k,M \in \N_0$, $M \geq k$,
	\begin{align*}
	W(\tilde{x}_{k+M+1}) - W(\tilde{x}_k) \leq \; &- \sum_{i=k}^{M} c_k l(\tilde{x}_i,\mu(\tilde{x}_i)) \\
	= \; & - \sum_{i=k}^{M} l(\tilde{x}_i,\mu(\tilde{x}_i)) - \sum_{i=k}^{M} (\Delta l)_i.
	\end{align*}
	With $\tilde{x}_{k+M+1} \ra \bar{x}$ and thus $W(\tilde{x}_{k+M+1}) \ra 0$ for $M \ra \infty$, and $k=0$,
	\begin{align*}
	\sum_{i=0}^{\infty} l(\tilde{x}_i,\mu(\tilde{x}_i)) &\leq W(\tilde{x}_0) - \sum_{i=0}^{\infty} (\Delta l)_i \\
	\Leftrightarrow \quad \qquad \underbrace{\sum_{i=0}^{\infty} l(\tilde{x}_i,\mu(\tilde{x}_i))}_{< \infty} &\leq V^{\infty}(x_0) + \delta_0 - \sum_{i=0}^{\infty} (\Delta l)_i,
	\end{align*}
	using $\delta_0$ as defined. 
	By optimality, $V^{\infty}(x_0)$ is smaller or equal the value on the left-hand side in the above inequality. 
	Therefore, the sum over $(\Delta l)_i$ is bounded by $\delta_0$ as
	\vspace*{-0.5em}
	\begin{align*}
	0 \leq \sum_{i=0}^{\infty} l(\tilde{x}_i,\mu(\tilde{x}_i)) - V^{\infty}(x_0) & \leq  \delta_0 - \sum_{i=0}^{\infty} (\Delta l)_i \\
	\Ra \hspace{3cm} \sum_{i=0}^{\infty} (\Delta l)_i & \leq \delta_0  . 
	\end{align*}
	Since $\tilde{x}_i \ra \bar{x}$, $l(\tilde{x}_{i},\mu(\tilde{x}_{i})) \ra 0$ for $i \rightarrow \infty$, and thus $(\Delta l)_i \ra 0$. 
	With the boundedness of $c_i \leq \bar{c}$, the limit of the sequence of partial sums of the left-hand side in the above inequality is a real number and the infinite series is convergent and finite. 
\end{proof}

If $W(x) \equiv V^{\infty}(x)$, \eqref{eq:decay-modified} changes to an equation with $c_k=1$, for all $k \in \N_0$, and thus $(\Delta l)_k = \delta_0 = 0$. 
Assuming that (in general) $\delta_0 \geq 0$, \eqref{eq:sum-Delta-c-bound} motivates to search for such functions that yield \eqref{eq:decay-modified} with $c_k \geq 1$ values. 	
Roughly speaking, the above states that the initial discrepancy $\delta_0>0$ may be partially compensated for by increasing the decay. 

\subsection{Setup of the Receding Horizon Scheme}

For a given sequence $c_k(i) \in \R_{>0}$, $i=0,\dotsc,N-1$, define the following value function at time step $k$ for a state $x_k^\ast \in \X$
\begin{subequations} \label{eq:MPC-2}
	\begin{align}
	\tilde{V}^{\infty}(x_k^\ast,k)& :=   \nonumber \\
	\min_{\{u_k(i)\}_{i=0}^{N-1}} \; &\sum_{i=0}^{N-1} c_{k}(i)l(x_k(i),u_k(i)) + F(x_k(N)) \\
	\text{s.t.} \quad &x_k(i+1)=f(x_k(i),u_k(i)), \; x_k(0)=x_k \\
	\quad &x_k(i) \in \X, \; u_k(i) \in \U \label{eq:MPC-2-xu-constr}\\
	\quad &x_k(N) \in \X_f \subset \X \label{eq:MPC-2-xstab-constr}.
	\end{align}
\end{subequations}
The state sequence associated with the minimizing control sequence $\{u_k^\ast (0), \dotsc,u_k^\ast (N-1)\}$ to \eqref{eq:MPC-2} is denoted by $\{x_k^\ast (1), \dotsc,x_k^\ast (N)\}$, whereas $x_{k+1}^\ast = x_{k+1} = f(x_k,u_k^\ast(0))$ is the state along \eqref{eq:sys} under the first element of the control sequence, with $x_{k}^\ast = x_k = x_k^\ast(0)$. 
It is assumed in the following that such a minimizing sequence to \eqref{eq:MPC-2} exists at time $k=0$. 
Due to the time-invariance of the dynamics \eqref{eq:sys} as well as the constraints \eqref{eq:MPC-2-xu-constr} and \eqref{eq:MPC-2-xstab-constr}, the control action is recursively feasible for subsequent times \cite{Mayne2000}.

It can be shown that a decay of the form
\vspace*{-0.4em}
\begin{align} \label{eq:MPC-2-cost-decay-c}
\forall k \in \N_0: \; \Delta_{u_k^\ast(0)} \tilde{V}^{\infty}(x_k^{\ast},k) \leq - c_k(0)l(x_k^{\ast},u_k^\ast(0)),
\end{align}
can be obtained through one inequality constraint over all coefficients in \eqref{eq:MPC-2}. 
That is, the coefficients must lie in the set $\mathcal{W}_{k-1} \ni c_k(i)$, $i=0,\dotsc,N-1$, defined by
\begin{align} \label{eq:MPC-2-cost-decay-c-constr}
\begin{split}
&\mathcal{W}_{k-1} :=\{ \;  C(i) \in \R_{>0}, i=0,\dotsc,N-1:\; \\
& \sum_{i=1}^{N-1} \left( C(i-1) - c_{k-1}(i) \right) l(x_{k-1}^\ast(i),u_{k-1}^\ast(i)) \\
&\; + F(f(x_{k-1}^\ast(N),\mu_F(x_{k-1}^\ast(N)))) - F(x_{k-1}^\ast(N))  \\
&\; + C(N-1)l(x_{k-1}^\ast(N),\mu_F(x_{k-1}^\ast(N))) \; \leq \; 0 \; \, \},
\end{split}
\end{align}
which contains the coefficients, controls and states to \eqref{eq:MPC-2} of preceding time step $k-1$ (associated with $\tilde{V}^{\infty}(x_{k-1}^{\ast},k-1)$, and thus denoted with the suffix $k-1$). 

For $k=0$, certain initial coefficients $c_0(i)\in \R_{>0}$ are used in $\tilde{V}^{\infty}(x_0,0)$, which may be determined by some offline optimization or chosen ``suitably'' (see Section \ref{subsec:init} for a discussion). 

\begin{prp}
	Let Assumption \ref{asm:local-ctrl} hold. 
	Let $c_0(i) \in \R_{>0}$, $i=0,\dotsc,N-1$.  
	If for all $k \in \N$ and all $i \in \{0,\dotsc,N-1\}$, $c_k(i) \in \mathcal{W}_{k-1}$, then \eqref{eq:MPC-2-cost-decay-c} holds for all $k \in \N_0$.
\end{prp}

\begin{proof}
	Consider the inequality \eqref{eq:MPC-2-cost-decay-c-constr} in $\mathcal{W}_{k-1}$ at time $k-1 \mapsto k$. 
	Adding $\sum_{i=1}^{N-1} c_{k}(i)l(x_k^\ast(i),u_k^\ast(i))$ and $F(x_k^\ast(N))$ to both sides reads 
	\begin{align*}
	&\sum_{i=1}^{N-1} c_{k+1}(i-1) l(x_k^\ast(i),u_k^\ast(i)) \\
	&+ \, c_{k+1}(N-1)l(x_{k}^\ast(N),\mu_F(x_{k}^\ast(N))) \\
	&+ \, F(f(x_{k}^\ast(N),\mu_F(x_{k}^\ast(N)))) \\
	\leq & \sum_{i=1}^{N-1} c_{k}(i) l(x_k^\ast(i),u_k^\ast(i)) + F(x_{k}^\ast(N)) \\
	= \, & \tilde{V}^{\infty}(x_k^{\ast},k) - c_{k}(0)l(x_{k}^\ast(0),u_{k}^\ast(0)).
	\end{align*}
	Define the feasible sequences $\{\tilde{x}_{k+1}(0),\dotsc,\tilde{x}_{k+1}(N)\}:= \{ \{x_k^\ast (1), \dotsc,x_k^\ast (N) \}, f(x_k^\ast{N},\mu_F(x_k^\ast(N))) \}$ and $\{\tilde{u}_{k+1}(0),\dotsc,\tilde{u}_{k+1}(N-1)\}:= \{ u_k^\ast(1),\dotsc,u_k^\ast(N-1),\mu_F(x_k^\ast(N)) \}$. 
	Substituting this into the left-hand side of the inequality gives
	\begin{align*}
	&\tilde{V}^{\infty}(x_{k+1}^{\ast},k+1) \\
	\leq &\sum_{i=0}^{N-1} c_{k+1}(i) l(\tilde{x}_{k+1}(i),\tilde{u}_{k+1}(i)) + F(\tilde{x}_{k+1}(N))
	\end{align*}
	by optimality of $\tilde{V}^{\infty}(x_{k+1}^{\ast},k+1)$.
\end{proof}
\begin{rem}
	By the intermediate calculation of coefficients satisfying \eqref{eq:MPC-2-cost-decay-c-constr}, at each time $k$ \emph{after} solving \eqref{eq:MPC-2}, stability can be enforced using a local LF with any decay $\alpha_F$. 
	That is, no stage cost related minimum decay rate is required inside the terminal set $\X_f$. 
	This reduces Assumption \ref{asm:local-ctrl-2} to the more general Assumption \ref{asm:local-ctrl}. 
\end{rem}
Recall from \eg \cite{Khalil1996-nonlin-sys}, that for a time-variant LF function $W(x,k)$ with controller $\mu(x)$ there exist $\alpha_1,\alpha_2 \in \mathcal{K}_{\infty}$ and $\alpha_3 \in \mathcal{K}$ such that for all $x \in \X, \spc k \in \N_0$, $\alpha_1(\|x\|) \leq W(x,k) \leq \alpha_2(\|x\|)$ and $\Delta_{\mu(x)} W(x,k) \leq -\alpha_3(\|x\|)$. 
Thus, for stability, the value function in \eqref{eq:MPC-2} must be bounded, which is equivalent to bounded coefficients for non-zero and bounded stage costs.

\subsection{Coefficient Update}
\label{subsec:coeff-interp}

In order to obtain somewhat \emph{meaningful} coefficients $c_k(i)$ and respective update laws, aspects of Q-learning \cite{Watkins1992}, and more specifically the Q-function, can be transferred into the above context. 
That is, for updating the coefficients, considering algorithms from Q-learning may be beneficial since an adaptation to the optimal decay rate is sought that minimizes the difference between the terms in \eqref{eq:sum-Delta-c-bound}. 
The Q-function $\Qfnc: \X \times \U \rightarrow \R_{\geq 0}$, formally defined as in \eqref{eq:Q-fcn}, is commonly replaced by some parametric architecture $Q(x,u):=\langle w^\ast,\varphi(x,u) \rangle$, where $w^\ast \in \R^p$ are the parameters to a regressor $\varphi:\X \times \U \rightarrow \R^p$, comprising some basis functions. Since \cite{Lewis2009}
\vspace*{-0.3em}
\begin{align*}
\forall x \in \X: \;  V^{\infty}(x) = \min_{u \in \U} \; \Qfnc(x,u),
\end{align*}
the function $Q$ seeks to approximate the infinite-horizon optimal value. 

In order to transfer this approach to the (online) predictive scheme, the \emph{stage cost may be used as regressor}. 
The underlying hypothesis in choosing particularly $\varphi(x,u) \equiv l(x,u)$ \ie $Q(x,u) \equiv w^\ast \, l(x,u)$, is that the optimal value function $V^{\infty}$ may be expressed as a certain magnitude, or factor, of the stage cost $l$.

\textbf{Temporal difference update.} For computation of $w^\ast$, certain \emph{offline} iterative routines are available \cite{Bertsekas2012}, approximating $V^{\infty}$ a priori over samples in the state-input space. 
The temporal difference method \cite{Sutton1988} is one particular approach that can be used to compute $w^\ast$ by iterating through 
\begin{align} \label{eq:TD-iter-off}
\begin{split}
w_{j}= &\argmin_{C} \; \big(  l(x,u) - C l(x,u) \\[-0.5em] 
&\hspace{2.6cm}+ \min_{a \in \U} \; w_{j-1} l(f(x,u),a) \big)^2,
\end{split}
\end{align} 
$j=0,1,\dotsc$, until $\|w_{j}-w_{j-1}\| \leq \eps_w$ for all $x,u$ (samples) and some $\eps_w>0$. 
Then, take $w^\ast := w_{j}$. 

\begin{rem}
	The above update rule refers to an off-policy learning as ``hypothetical'' actions $a$ are compared to find the respective minimum \cite{Singh2000}. 
		In the presented case, however, sample control actions are not necessary as the minimum of $l$ over $u$ is at $a = \bar{u}$ for all times. 
\end{rem}

For \emph{online} adaptation, however, the parameters must be updated using only the current state and control obtained through \eqref{eq:MPC-2}. 
This can be achieved using \eqref{eq:TD-iter-off} for $j=k$ with $x=x_k$ and $u=u_k$. 
In this case, observe that the updated parameter correlates to the coefficient used in the subsequent time step in \eqref{eq:MPC-2} \ie $w_{k} = c_{k+1}$, $\forall k \in \N_0$. 

\begin{rem}
	There exists approaches which handle the issue of providing samples online and utilize parameters $w \in \R^p$, $p>1$. 
		In \cite{Lewis2009}, the use of recursive least-squares is suggested, updating the parameters each $p$ time steps, using the past $p$ state-action values. 
		Provided, such ``buffer'' can be used, the parameters can be updated by sampling a minibatch uniformly from that buffer \cite{Lillicrap2016}. 
\end{rem}	

\begin{rem}
	One should be aware of the distinct $TD(\lambda)$ \cite{Sutton1988,Dayan1992} approaches, representing a relevance weighting $0 \leq \lambda \leq 1$ of recent temporal differences to earlier ones \cite{Rummery1994}. 
\end{rem}

In the case $p=1$, the analogue online update rule to \eqref{eq:TD-iter-off} for any of the $N$ coefficients may then read as
\begin{align} \label{eq:TD-iter-on}
\begin{split}
c_{k+1}(i)  = &\arginf_{C \in \R_{>0}} \; \big(  l(x_k^\ast(i),u_k^\ast(i)) - C l(x_k^\ast(i),u_k^\ast(i)) \\[-0.5em] 
& \hspace{1.8cm}+ c_{k}(i) \min_{a \in \U} \; l(f(x_k^\ast(i),u_k^\ast(i)),a) \big)^2, 
\end{split}
\end{align}
$i=0, \dotsc,N-1$. 
It needs to be pointed out that different update rules may be used for the individual $c_k(i)$. 
Note that the coefficient is pulled out of the minimization objective in the last term on the right-hand side in \eqref{eq:TD-iter-on} as it does not change the minimum (which is also true for \eqref{eq:TD-iter-off}). 
The infimum is used since the feasible set is open. 

As the stability constraint \eqref{eq:MPC-2-cost-decay-c-constr} must be satisfied, certain modifications to the respective update rule are necessary. 
For instance, in the optimization \eqref{eq:TD-iter-on}, which for all coefficients $C^{+} = [c_{k+1}(0),\dotsc,c_{k+1}(N-1)]^\top$ reads as
\vspace*{-0.3em}
\begin{align} \label{eq:critic-prob}
C^{+}= \arginf_{[C_0,\dotsc,C_{N-1}]^\top \in \R_{>0}^N} \; \sum_{i=0}^{N-1} (\beta_i - C_i l(x_k^\ast(i),u_k^\ast(i)))^2
\end{align}
with $\beta_i = l(x_k^\ast(i),u_k^\ast(i)) + c_{k}(i)l(f(x_k^\ast(i),u_k^\ast(i)),\bar{u})$, the set \eqref{eq:MPC-2-cost-decay-c-constr} serves as optimization constraint. 

For $k=0$, an initial guess of coefficients $c_{0}(i)$, $i=0,\dotsc,N-1$, is required, which is a crucial step in the sense that it may influence the performance severely. 
Unfortunately, no rule for choosing $c_0$ can be derived to this point (see also Section \ref{subsec:init}). 

Summarizing, the algorithm is depicted in Table \ref{tab:algorithm}. 
\vspace*{-0.5em}
\begin{table}[h]
	\begin{tabular}{p{0.455\textwidth}}
		\hline
		\textbf{Algorithm:}\\
		\hline 
		\begin{enumerate}
			\setlength{\itemsep}{1pt}
			\item[1)] Initialize:  Choose $F$, $\mu$, $\X_f$ and $c_0(i) \in \R_{>0}$, $i=0,\dotsc,N-1$
		\end{enumerate}
		Set $k=0$. Then,
		\begin{enumerate}
			\setlength{\itemsep}{1pt}
			\item[2)] Obtain state $x_k$.
			\item[3)] Solve \eqref{eq:MPC-2} using $c_k(i)$, $i=0,\dotsc,N-1$, to obtain $\{u_k^\ast(i)\}_{i=0}^{N-1}$ and $\{x_k^\ast(i)\}_{i=0}^{N}$
			\item[4)] Solve \eqref{eq:critic-prob} under the constraint set $ \mathcal{W}_k$ to obtain $c_{k+1}(i) \in \mathcal{W}_k$, $i=0,\dotsc,N-1$
			\item[5)] Apply the first element of the sequence $u_k^\ast(0)$ to \eqref{eq:sys}
			\item[6)] $k \mapsto k+1$, go to 2)
		\end{enumerate}\\
		\hline 
	\end{tabular}
	\caption{Procedure of the online scheme.}
	\label{tab:algorithm}
	\vspace*{-2em}
\end{table}

\textbf{Allocation.} As an alternative to an additional online optimization as in \eqref{eq:critic-prob}, the coefficients may be specified as the following: 
Looking at the set $\mathcal{W}_{k}$, one can choose
\begin{align} \label{eq:c-constr-equal-newold}
c_{k+1}(i-1) = c_k(i), \quad i=1,\dotsc,N-1.
\end{align}
Then, to ensure the inequality \eqref{eq:MPC-2-cost-decay-c-constr} in step $k$, 
\begin{align} \label{eq:c-constr-inequal-last}
\begin{split}
&c_{k+1}(N-1) l(x_k^\ast(N),\mu_F(x_k^\ast(N)) \\
\leq\;  &F(x_k^\ast(N)) - F \left( f \left( x_k^\ast(N),\mu_F(x_k^\ast(N)) \right) \right),
\end{split}
\end{align}
which, under Assumption \ref{asm:local-ctrl-2}, may be restated as the specific allocation
\begin{align} 
c_{k+1}(N-1) &= \, \frac{F(x_k^\ast(N)) - F \left( f \left( x_k^\ast(N),\mu_F(x_k^\ast(N)) \right) \right)}{l(x_k^\ast(N),\mu_F(x_k^\ast(N))}  \nonumber \\
&= \,  \bar{\alpha}_k. \label{eq:c-constr-equal-last}
\end{align} 

\begin{prp} \label{prop:c-bound-alloc}
	Let Assumption \ref{asm:local-ctrl-2} hold. 
	If the initial coefficients are $c_0(i) = 1$, for all $i \in \{0,\dotsc,N-1\}$, then the following bound regarding the decay rate \eqref{eq:MPC-2-cost-decay-c} holds:
	\begin{align}
	\forall k \in \N_0: \; 1 \leq c_k(0) \leq \bar{c} = \sup_{k \in \N_0} \; \bar{\alpha}_k < \infty.
	\end{align}
\end{prp}

\begin{proof}
	The result follows directly from \eqref{eq:c-constr-equal-newold} and \eqref{eq:c-constr-equal-last}.
\end{proof}

Observe that for $F$ fulfilling Assumption \ref{asm:local-ctrl}, smaller decay rates can be obtained in \eqref{eq:MPC-2-cost-decay-c} than under Assumption \ref{asm:local-ctrl-2} due to the different local decay rates of $F$. 
It should be noted, however, that the above allocation does not represent any learning strategy but mainly addresses the issue of maintaining stability while providing a specific coefficient upper bound which can be used in Proposition \ref{prop:cost-rel}.

\section{Discussion}
\label{sec:prop-ana}

In the following, the temporal difference update is examined under the assumption that the stability related constraint \eqref{eq:MPC-2-cost-decay-c-constr} is satisfied while furthermore the initialization of the scheme is discussed. 

\subsection{Online Adaptation}
\label{subsec:online-adapt}

To elaborate on the meaning of the update \eqref{eq:TD-iter-on} for the individual coefficients, it is assumed that, under \eqref{eq:TD-iter-on}, the stability constraint is satisfied for all times: 
\begin{asm} \label{asm:TDonline-stab}
	For any $k \in \N_0$, under the sequences $\{x_k^\ast(i)\}_{i=0}^{N}$ and $\{u_k^\ast(i)\}_{i=0}^{N-1}$ from \eqref{eq:MPC-2}, the coefficients $c_{k+1}(i)$, $i=0,\dotsc,N-1$, obtained through \eqref{eq:TD-iter-on} lie in $\mathcal{W}_k$.
\end{asm}

This is purely to demonstrate certain properties of the update and does not affect the functionality of the algorithm. 
As mentioned previously, it is desirable to have $(\Delta l)_k$ in \eqref{eq:cost-rel-mod-decay-VN} as large as possible. 
Observe the following:

\begin{prp} \label{prop:c-online-increase}
	Let Assumption \ref{asm:local-ctrl} and \ref{asm:TDonline-stab} hold. 
	Let $x_k^\ast(i)$ and $u_k^\ast(i)$, $i=0,\dotsc,N-1$ be the solution to \eqref{eq:MPC-2} for any $k \in \N_0$. 
	If, for any $k \in \N_0$, $c_k(i) \in (0, 1]$ and $(x_k^\ast(i),u_k^\ast(i)) \neq (\bar{x},\bar{u})$, then
	\vspace*{-0.4em}
	\begin{align} \label{eq:Delta-c-incr}
	\Delta c_k(i) := c_{k+1}(i) - c_k(i) > 0,
	\end{align}
	under \eqref{eq:TD-iter-on}, for any $i \in \{0,\dotsc,N-1\}$. 
\end{prp}

\begin{proof}
	If $c_k(i) \in (0,1]$ and $(x_k^\ast(i),u_k^\ast(i)) \neq (\bar{x},\bar{u})$, then \eqref{eq:TD-iter-on} yields
	\begin{align}
	c_{k+1}(i) l(x_k^\ast(i),u_k^\ast(i)) &- l(x_k^\ast(i),u_k^\ast(i)) \nonumber  \\[0.4em] 
	= & c_{k}(i)\min_{a \in \U} \; l(f(x_k^\ast(i),u_k^\ast(i)),a) \nonumber \\
	\Leftrightarrow \qquad c_{k+1}(i) = &1 + c_k(i)\frac{l(f(x_k^\ast(i),u_k^\ast(i)),\bar{u}) }{l(x_k^\ast(i),u_k^\ast(i))}, \label{eq:c-update-TDonline}
	\end{align} 
	where $\min_{a \in \U} \; l(f(x_k^\ast(i),u_k^\ast(i)),a) = l(f(x_k^\ast(i),u_k^\ast(i)),\bar{u})$. Since $(x_k^\ast(i),u_k^\ast(i)) \neq (\bar{x},\bar{u})$, the last term on the right-hand side is strictly positive, which implies that $c_k(i) \in (0,1] \, \Ra \, c_{k+1}(i) >1$, and subsequently \eqref{eq:Delta-c-incr}.
\end{proof}  

It can be deduced that whenever the algorithm is initialized with value $c_0(i) \in (0,1]$, $i=0,\dotsc,N-1$, all subsequent coefficients are $c_k(i) >  1$, $k \in \N$, if state and control satisfy $(x_k^\ast(i),u_k^\ast(i)) \neq (\bar{x},\bar{u})$.

\begin{rem} \label{rem:lower-bound-opt}
	By the above property of the coefficients tending toward higher magnitude and by definition of the Q-function being the infinite sum of stage cost, it is legitimate to directly impose a lower bound constraint $c_k(i) \geq 1$, $i \in \{0,\dotsc,N-1\}$, for any $k \in \N$, on \eqref{eq:critic-prob}. 	
\end{rem}

Observe, when $(x_k^\ast(i),u_k^\ast(i)) = (\bar{x},\bar{u})$, any $c_k(i) \in \R_{>0}$ is a feasible candidate to \eqref{eq:TD-iter-on}.  

Recalling from Proposition \ref{prop:cost-rel} that there exists an upper bound on the (first) parameter $c_k(0) \leq \bar{c}$, the following assumption regarding \eqref{eq:c-update-TDonline} is made: 
\begin{asm} \label{asm:bound-seq-limsup}
	For any sequences $\{x_k\}_{k=0}^{\infty}$ $\{u_k\}_{k=0}^{\infty}$, with $x_k \in \X$, $u_k \in \U$ and $(x_k,u_k) \ra (\bar{x},\bar{u})$ for $k \ra \infty$,
	\begin{align}
	\limsup_{(x_k,u_k) \ra (\bar{x},\bar{u})} \; \frac{l(f(x_k,u_k),0)}{l(x_k,u_k)} < \infty.
	\end{align}
\end{asm}

The above Assumption \ref{asm:bound-seq-limsup} can be verified only using properties of the dynamic $f$ and the stage cost $l$. 

\begin{rem}
	Assumption \ref{asm:bound-seq-limsup} does not influence the functionality of the proposed algorithm but only provides a bound $\bar{c}$ regarding the suboptimality estimate established in Proposition \ref{prop:cost-rel}. 
	Furthermore, referring to \eqref{eq:TD-iter-on} and \eqref{eq:critic-prob} -- for bounded, nonzero $(x_k^\ast(i),u_k^\ast(i))$ and bounded $0 <c_k(i)<\infty$ -- an infimum is attained for bounded $c_{k+1}(i) < \infty$.
\end{rem} 

Hence, different upper bounds $\bar{c}$ are obtained by Proposition \ref{prop:c-bound-alloc} and by the supremum of the right-hand side of Eq. \eqref{eq:c-update-TDonline} according to the value in Assumption \ref{asm:bound-seq-limsup}.
This implies that different performances may be achieved with respect to the estimate \eqref{eq:cost-rel-mod-decay-VN}, specifically regarding the difference between $\delta_0$ and the sum over $(\Delta l)_k$.

\subsection{Initialization}
\label{subsec:init}

Notice, that in \eqref{eq:cost-rel-mod-decay-VN}, $W(x_k)$ relates to $\tilde{V}^{\infty}(x_k^\ast,k)$ whereas the applied control is given by $\mu(x_k) = u_k^\ast(0)$. 
Since, if Assumption \ref{asm:local-ctrl-2} holds, $V^{\infty}(x_0) \leq \tilde{V}_{\text{MPC}}^{\infty}(x_0)$ for any $x_0 \in \X$, inequality \eqref{eq:cost-rel-mod-decay-VN} can be extended to give an indirect comparison with the MPC cost. Denoting $\gamma_0 := \tilde{V}_{\text{MPC}}^{\infty}(x_0) - V^{\infty}(x_0)$, it follows that under Assumption \ref{asm:local-ctrl-2}, $\gamma_0 \geq 0$ and furthermore
\vspace*{-0.5em}
\begin{align} \label{eq:comparison-MPC1-MPC2}
\sum_{k=0}^{\infty} l( x_k^\ast,u_k^\ast(0) ) \leq \tilde{V}_{\text{MPC}}^{\infty}(x_0) \underbrace{ -\gamma_0 + \delta_0 - \sum_{k=0}^{\infty} (\Delta l)_k}_{:= \Delta V_{N}},
\end{align}  
where $x_0^\ast = x_0 \in \X$. 
This implies that if $\gamma_0$ is large enough and the infinite sum over $(\Delta l)_k$ has a sufficiently high (finite) value, $\Delta V_{N}$ can be made non-positive and thus
\vspace*{-0.5em}
\begin{align*}
\sum_{k=0}^{\infty} l( x_k^\ast,u_k^\ast(0) ) \leq \tilde{V}_{\text{MPC}}^{\infty}(x_0).
\end{align*}
This comparison, however, is in so far indirect as it does not involve the actual infinite sum of stage costs under the control actions generated through \eqref{eq:MPC-1} (with Assumption \ref{asm:local-ctrl-2}). 
Observe that this sum, by optimality, must be larger than $V^{\infty}$ from \eqref{eq:cost-opt-hor-inf} and thus from \eqref{eq:cost-rel-mod-decay-VN},
\vspace*{-0.5em}
\begin{align} \label{eq:comparison-MPC1-MPC2-direct}
\sum_{k=0}^{\infty} l( x_k^\ast,u_k^\ast(0) ) \leq \sum_{k=0}^{\infty} l( x_k^{\text{MPC}},u_k^{\text{MPC}}(0) ) + \delta_0 - \sum_{k=0}^{\infty} (\Delta l)_k,
\end{align} 
with $\{x_k^{\text{MPC}} \}_{k=0}^{\infty}$ and $\{u_k^{\text{MPC}}(0)\}_{k=0}^{\infty} \in \U_{\text{ad}} $ the state and control trajectory, respectively, under \eqref{eq:MPC-1}. 
Here, only limited information about the relationship between $V^{\infty}$ and the first term on the right-hand side of \eqref{eq:comparison-MPC1-MPC2-direct} is available though. 

Due to (in general) missing information about the solution to the IHOCP \eqref{eq:cost-opt-hor-inf}, no specific allocation rule for the coefficients in \eqref{eq:MPC-2} at $k=0$ can be given. 
Therefore, starting the scheme \eqref{eq:MPC-2} equivalently to MPC \eqref{eq:MPC-1} \ie $c_0(i)=1$, $i=0,\dotsc,N-1$, is reasonable. 
Yet, as elaborated in Proposition \ref{prop:c-online-increase}, the coefficients may tend to a higher value, which suggests an initialization $c_0(i)>1$.

\section{Simulation Study}
\label{sec:simstudy}

Next, the time-variant stage cost approach \eqref{eq:MPC-2} is compared to the standard MPC scheme \eqref{eq:MPC-1} using two test systems. 

First, a linear mass spring damper system 
\begin{align} \label{eq:sys-ex-1}
m\ddot{x}(t) + d \dot{x}(t) + s x(t) = u 
\end{align}
with specific $m=1\spc \text{Kg}$, $d=0.2\spc \frac{\text{Ns}}{\text{m}}$ and $s=1\spc \frac{\text{N}}{\text{m}}$ is used under Euler discretization with sampling time $\Delta t = 0.7$ (exact discretization is not used for consistency with the second example).
The stage cost and the horizon are set to $l(x,u) = \|x\|^2 + u^2$ and $N = 10$, respectively. 

Secondly, the nonlinear system from \cite{Primbs1999} 
\begin{align} \label{eq:sys-ex-2}
\footnotesize
\begin{split}
\dot{x}_1(t) = &\;x_2(t) \\
\dot{x}_2(t) = &-x_1(t) \left(\frac{\pi}{2} + \arctan(5x_1(t))  \right) - \frac{5 x_1^2(t)}{2(1+25x_1^2(t))}\\
&+ 4x_2(t) + 3u
\end{split}
\end{align}
is Euler discretized using $\Delta t = 0.1$. 
The stage cost and the horizon are set to $l(x,u) = x_2^2 + u^2$ and $N = 5$, respectively.

\begin{figure}[h]
	\centering
	\includegraphics[width=0.8\columnwidth]{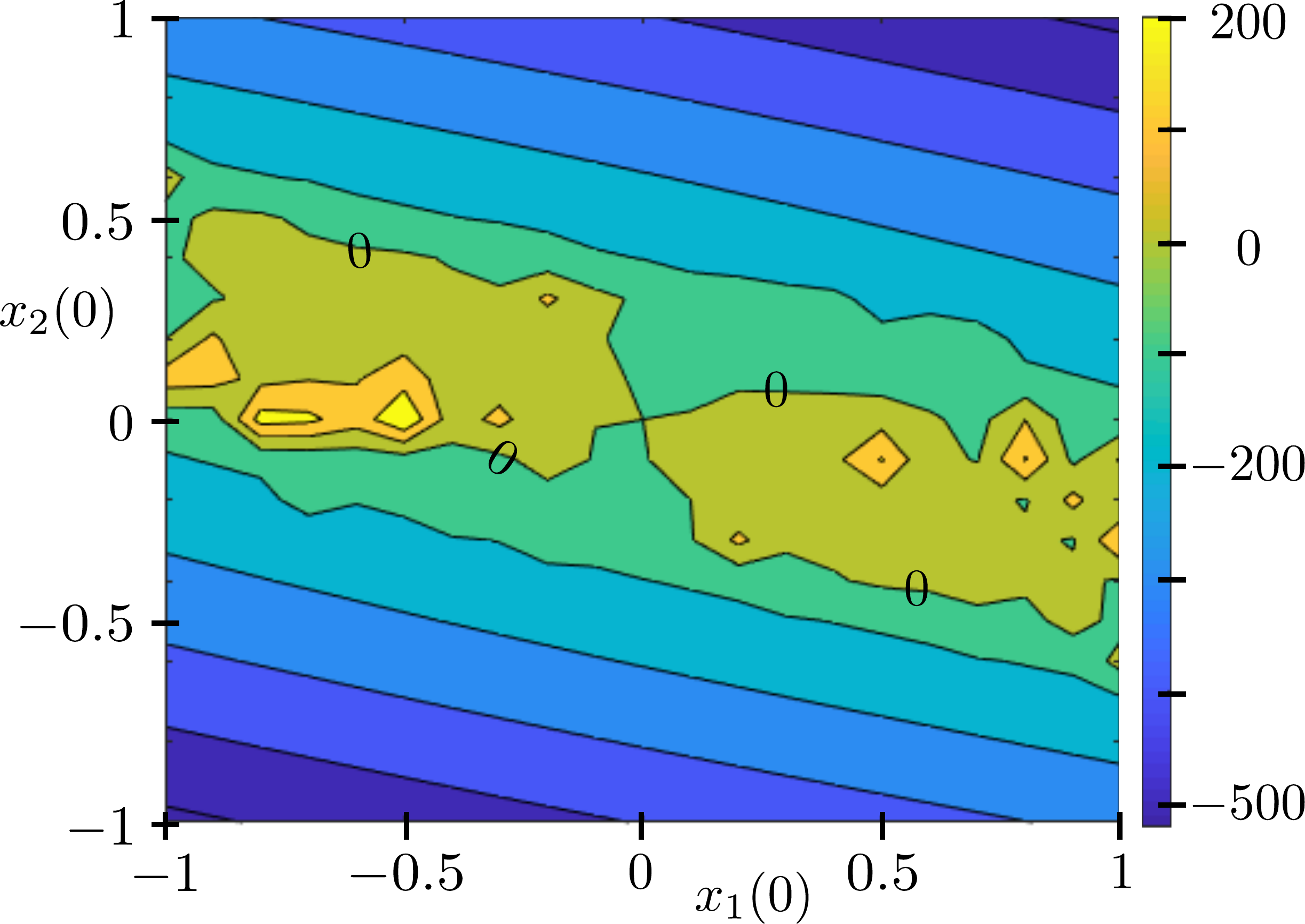}
	\vspace*{-0.5em}
	\caption{State space $[-1,1]^2$ of initial states $x_0$ for system \eqref{eq:sys-ex-1}. 
		For certain areas of initial states, the infinite-horizon cost under \eqref{eq:MPC-2} is lower than under \eqref{eq:MPC-1} (blue), whereas elsewhere it may be higher (yellow).
	}
	\label{fig:mass_spring_damp}
		\vspace*{-3mm}
\end{figure}

\begin{figure}[h]
	\centering
	\includegraphics[width=0.8\columnwidth]{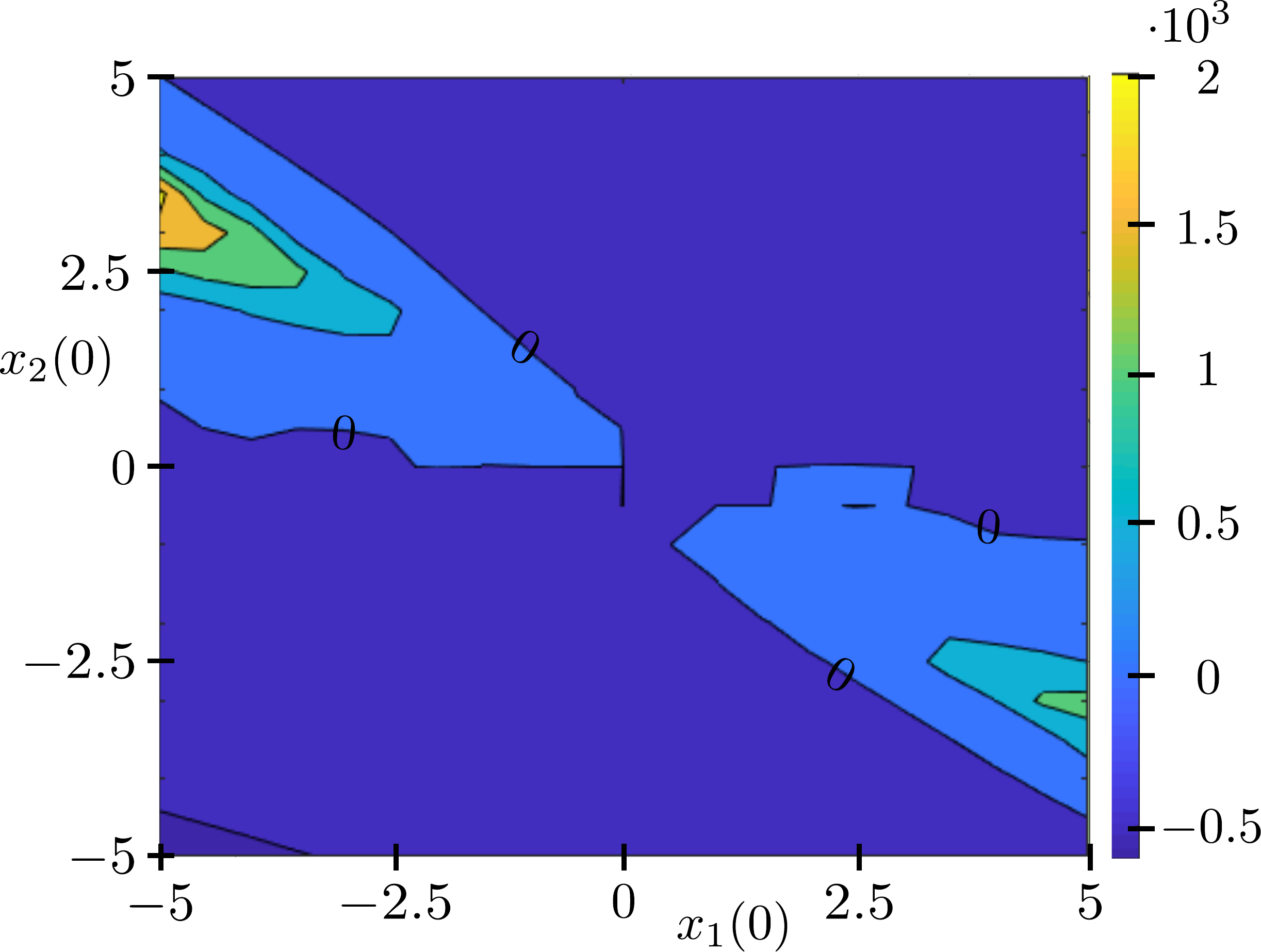}
		\vspace*{-0.5em}
	\caption{State space $[-5,5]^2$ of initial states $x_0$ for system \eqref{eq:sys-ex-2}. 
		In particular, some symmetry can be detected in the performance under either control scheme. 
		Whereas some state space segments mark better performance under MPC, adapting the coefficients appears to be beneficial for a larger partition of the space (dark blue). }
	\label{fig:primbs}
		\vspace*{-6mm}
\end{figure}

For linear and nonlinear system control, initial sets of coefficients $c_0(i)=5$, $i\in \{0,\dotsc,9\}$ and $c_0(i) = 20$, $i \in \{0,\dotsc,4\}$, are chosen assuming no further knowledge about the optimal value function. 
Under variation of the initial discrepancies $-\gamma_0+\delta_0$ in \eqref{eq:comparison-MPC1-MPC2} by using other initial coefficients, the obtained simulation results may be different. 
The terminal costs (and a corresponding controller) are obtained by finding a local LQ control and by using the optimal (continuous time) value function given in \cite{Primbs1999}, respectively, for the two systems. 
The terminal set was obtained using level sets $\X_f := \{x \in \X: \, F(x) \leq 0.1\}$.
Regarding \eqref{eq:comparison-MPC1-MPC2}, a comparison of MPC \eqref{eq:MPC-1} and the proposed scheme \eqref{eq:MPC-2} can be seen in Fig. \ref{fig:mass_spring_damp} as well as in Fig. \ref{fig:primbs} for either system. 
The scale on the right side of the contour depicts the difference between the (quasi-) infinite sum of stage costs under either algorithm: positive values denote superiority of MPC and negative values indicate better performance under \eqref{eq:MPC-2}, using the presented initial coefficients. 
Though the performances are dependent on a, roughly speaking, suitable initialization of the coefficients, they may also vary with the specific coefficient update rule used \eg either by optimization or allocation. 
In the simulation, the algorithm presented in Table \ref{tab:algorithm} is used. 
Additionally, for the nonlinear system, a lower bound on $c_k(i) \geq 1$, for all $i \in \{0,\dotsc,4\}$, was imposed on the optimization in \eqref{eq:critic-prob} according to Remark \ref{rem:lower-bound-opt}. 

It can be seen in the proposed approach may outperform MPC for certain initial states with the given configuration, whereas elsewhere MPC appears superior. 
	As pointed out in Sec. \ref{subsec:init}, the initial coefficients are, to this point, chosen arbitrarily, the performance under \eqref{eq:MPC-2} is not guaranteed to be better with respect to the IHOCP. 
	However, the results indicate the potential of cost reduction when the stage cost is weighted in the online optimization. 
	This may be of special interest for short predictive horizon problems, as then the computational load is within limits (also regarding the additional optimization of the coefficients). 


\section{Conclusion}

In this work, the rather general question of MPC performance with respect to a nonlinear IHOCP problem is addressed. 
In the framework of stabilizing MPC, specifically that using a stabilizing terminal set constraint and terminal cost, some properties of a modified finite-horizon cost function with corresponding decay are presented and put in relation with the solution to the IHOCP. 
It is argued that an approximation of the IHOCP by the proposed finite-horizon scheme can induce a higher decay rate which could subsequently reduce the ``degree of suboptimality''. 
Through introduction of coefficients to the standard MPC problem, this decay rate can be achieved via an additional optimization problem while specific allocation can also be used to circumvent online optimization and guarantee stability.

\bibliographystyle{plain}
\bibliography{bib/constructing-LFs,bib/discont-DE,bib/stabilization,bib/stability,bib/non-smooth-analysis,bib/sliding-mode,bib/MPC,bib/opt-ctrl,bib/dyn-sys,bib/ADP_RL}

\begin{thebibliography}{10}

\bibitem{Bertsekas2012}
D.~P. Bertsekas.
\newblock {\em Dynamic Programming and Optimal Control: Approximate Dynamic
  Programming}, volume~II.
\newblock Athena Scientific, 4th edition, 2012.

\bibitem{Chen1998}
H.~Chen and F.~Allg{\"o}wer.
\newblock A quasi-infinite horizon nonlinear model predictive control scheme
  with guaranteed stability.
\newblock {\em Automatica}, 34(10):1205--1217, 1998.

\bibitem{Dayan1992}
P.~Dayan.
\newblock The convergence of {T}{D}($\lambda$) for general $\lambda$.
\newblock {\em Machine Learning}, 8(3-4):341--362, 1992.

\bibitem{Fontes2001}
F.~A. C.~C. Fontes.
\newblock A general framework to design stabilizing nonlinear model predictive
  controllers.
\newblock {\em Syst. Control Lett.}, 42(2):127--143, 2001.

\bibitem{Gruene2009}
L.~Gr\"{u}ne.
\newblock {A}nalysis and design of unconstrained nonlinear {M}{P}{C} schemes
  for finite and infinite dimensional systems.
\newblock {\em SIAM J. Control Optimiz.}, 48(2):1206--1228, 2009.

\bibitem{Gruene2008}
L.~Gr\"{u}ne and A.~Rantzer.
\newblock {O}n the infinite horizon performance of receding horizon
  controllers.
\newblock {\em IEEE Trans. Automat. Control}, 53(9):2100--2111, 2008.

\bibitem{Hu2002}
B.~Hu and A.~Linnemann.
\newblock Toward infinite-horizon optimality in nonlinear model predictive
  control.
\newblock {\em IEEE Trans. Automat. Control}, 47(4):679--682, 2002.

\bibitem{Jadbabaie2005}
A.~Jadbabaie and J.~Hauser.
\newblock On the stability of receding horizon control with a general terminal
  cost.
\newblock {\em IEEE Trans. Automat. Control}, 50(5):674--678, 2005.

\bibitem{Jadbabaie2001a}
A.~Jadbabaie, J.~A. Primbs, and J.~Hauser.
\newblock Unconstrained receding horizon control with no terminal cost.
\newblock In {\em Proc. American Control Conference}, 2001.

\bibitem{Khalil1996-nonlin-sys}
H.~Khalil.
\newblock {\em Nonlinear {S}ystems}.
\newblock Prentice-Hall. 2nd edition, 1996.

\bibitem{Lewis2009}
F.~L. Lewis and D.~Vrabie.
\newblock Reinforcement learning and adaptive dynamic programming for feedback
  control.
\newblock {\em IEEE Circuits Syst. Mag.}, 9(3):32--50, 2009.

\bibitem{Lillicrap2016}
T.~P. Lillicrap, J.~J. Hunt, A.~Pritzel, N.~Heess, T.~Erez, Y.~Tassa,
  D.~Silver, and D.~Wierstra.
\newblock Continuous control with deep reinforcement learning.
\newblock Available at arXiv:1509.02971v5 [cs.LG], 2016.

\bibitem{Liu2012}
D.~Liu, D.~Wang, D.~Zhao, Q.~Wei, and N.~Jin.
\newblock Neural-network-based optimal control for a class of unknown
  discrete-time nonlinear systems using globalized dual heuristic programming.
\newblock {\em IEEE Trans. Automat. Sc. and Eng.}, 9(3):628--634, 2012.

\bibitem{Mayne2000}
D.~Q. Mayne, J.~B. Rawlings, C.~V. Rao, and P.~O.~M. Scokaert.
\newblock {Constrained} model predictive control: {Stability} and optimality.
\newblock {\em Automatica}, 36(6):789--814, 2000.

\bibitem{Primbs1999}
J.~A. Primbs, V.~Nevisti\'{c}, and J.~C. Doyle.
\newblock {N}onlinear optimal control: {A} control {L}yapunov function and
  receding horizon perspective.
\newblock {\em Asian J. Control}, 1(1):14--24, 1999.

\bibitem{Reble2011}
M.~Reble and F.~Allg\"{o}wer.
\newblock Unconstrained nonlinear model predictive control and suboptimality
  estimates for continuous-time systems.
\newblock In {\em Proc. 18th IFAC World Congress}, 2011.

\bibitem{Rosolia2018}
U.~Rosolia and F.~Borrelli.
\newblock {L}earning {M}odel {P}redictive {C}ontrol for {I}terative {T}asks.
  {A} {D}ata-{D}riven {C}ontrol {F}ramework.
\newblock {\em IEEE Trans. Automat. Control}, 63(7):1883--1896, 2018.

\bibitem{Rummery1994}
G.~A. Rummery and M.~Niranjan.
\newblock {\em On-line Q-learning using connectionist systems}, volume~37.
\newblock University of Cambridge, Department of Engineering Cambridge,
  England, 1994.

\bibitem{Singh2000}
S.~Singh, T.~Jaakkola, M.~L. Littman, and C.~Szepesv{\'a}ri.
\newblock {C}onvergence {R}esults for {S}ingle-{S}tep {O}n-{P}olicy
  {R}einforcement-{L}earning {A}lgorithms.
\newblock {\em Machine Learning}, 38(3):287--308, Mar 2000.

\bibitem{Sutton1988}
R.~S. Sutton.
\newblock Learning to predict by the methods of temporal differences.
\newblock {\em Machine Learning}, 3(1):9--44, 1988.

\bibitem{Tuna2006}
S.~E. Tuna, M.~J. Messina, and A.~R. Teel.
\newblock Shorter horizons for model predictive control.
\newblock In {\em Proc. American Control Conference}, 2006.

\bibitem{Watkins1992}
C.~Watkins and P.~Dayan.
\newblock Q-learning.
\newblock {\em Machine Learning}, 8(3-4):279--292, 1992.
\newblock Technical Note.

\bibitem{Zanon2019}
M.~Zanon and S.~Gros.
\newblock Safe reinforcement learning using robust {M}{P}{C}.
\newblock Available at arXiv:1906.04005v1 [cs.SY], 2019.

\end{thebibliography}

\end{document}